\documentclass[12pt]{amsart} 
\usepackage{amsmath,amsthm,amsfonts,amssymb,latexsym}

\begin{document}

\theoremstyle{plain}

\newtheorem{thm}{Theorem}[section]
\newtheorem{lem}[thm]{Lemma}
\newtheorem{pro}[thm]{Proposition}
\newtheorem{cor}[thm]{Corollary}
\newtheorem{que}[thm]{Question}
\theoremstyle{definition}\newtheorem{rem}[thm]{Remark}
\theoremstyle{definition}\newtheorem{rems}[thm]{Remarks}
\theoremstyle{definition}\newtheorem{defi}[thm]{Definition}
\theoremstyle{definition}\newtheorem{Question}[thm]{Question}
\newtheorem{con}[thm]{Conjecture}

\newtheorem*{thmA}{Theorem A}
\newtheorem*{thmB}{Theorem B}
\newtheorem*{thmC}{Theorem C}

\newtheorem*{thmAcl}{Main Theorem$^{*}$}
\newtheorem*{thmBcl}{Theorem B$^{*}$}

\newcommand{\Maxn}{\operatorname{Max_{\textbf{N}}}}
\newcommand{\Syl}{\operatorname{Syl}}
\newcommand{\dl}{\operatorname{dl}}
\newcommand{\Con}{\operatorname{Con}}
\newcommand{\cl}{\operatorname{cl}}
\newcommand{\Stab}{\operatorname{Stab}}
\newcommand{\Aut}{\operatorname{Aut}}
\newcommand{\Ker}{\operatorname{Ker}}
\newcommand{\fl}{\operatorname{fl}}
\newcommand{\Irr}{\operatorname{Irr}}
\newcommand{\IBr}{\operatorname{IBr}}
\newcommand{\SL}{\operatorname{SL}}
\newcommand{\NN}{\mathbb{N}}
\newcommand{\N}{\mathbf{N}}
\newcommand{\C}{\mathbf{C}}
\newcommand{\OO}{\mathbf{O}}
\newcommand{\F}{\mathbf{F}}
\newcommand{\FF}{\mathbb{F}}
\newcommand{\CF}{\mathcal{F}}
\newcommand{\CO}{\mathcal{O}}

\renewcommand{\labelenumi}{\upshape (\roman{enumi})}

\newcommand{\PSL}{\operatorname{PSL}}
\newcommand{\PSU}{\operatorname{PSU}}

\providecommand{\V}{\mathrm{V}}
\providecommand{\E}{\mathrm{E}}
\providecommand{\ir}{\mathrm{Irr_{rv}}}
\providecommand{\Irrr}{\mathrm{Irr_{rv}}}
\providecommand{\re}{\mathrm{Re}}

\def\Z{{\mathbb Z}}
\def\C{{\mathbb C}}
\def\Q{{\mathbb Q}}
\def\irr#1{{\rm Irr}(#1)}
\def\ibr#1{{\rm IBr}(#1)}
\def\irrv#1{{\rm Irr}_{\rm rv}(#1)}
\def \c#1{{\cal #1}}
\def\cent#1#2{{\bf C}_{#1}(#2)}
\def\syl#1#2{{\rm Syl}_#1(#2)}
\def\nor{\triangleleft\,}
\def\oh#1#2{{\bf O}_{#1}(#2)}
\def\Oh#1#2{{\bf O}^{#1}(#2)}
\def\zent#1{{\bf Z}(#1)}
\def\det{\mathrm{det}}
\def\ker#1{{\rm ker}(#1)}
\def\norm#1#2{{\bf N}_{#1}(#2)}
\def\alt#1{{\rm Alt}(#1)}
\def\iitem#1{\goodbreak\par\noindent{\bf #1}}
   \def \mod#1{\, {\rm mod} \, #1 \, }
\def\sbs{\subseteq}

\def\gc{{\bf GC}}
\def\ngc{{non-{\bf GC}}}
\def\ngcs{{non-{\bf GC}$^*$}}
\newcommand{\notd}{{\!\not{|}}}
\def\OG{{\mathcal{O}G}}
\def\foc{\mathfrak{foc}}
\def\hyp{\mathfrak{hyp}}

\newcommand{\Out}{{\mathrm {Out}}}
\newcommand{\Br}{{\mathrm {Br}}}
\newcommand{\Mult}{{\mathrm {Mult}}}
\newcommand{\Inn}{{\mathrm {Inn}}}
\newcommand{\IBR}{{\mathrm {IBr}}}
\newcommand{\IBRL}{{\mathrm {IBr}}_{\ell}}
\newcommand{\IBRP}{{\mathrm {IBr}}_{p}}
\newcommand{\ord}{{\mathrm {ord}}}
\def\id{\mathop{\mathrm{ id}}\nolimits}
\renewcommand{\Im}{{\mathrm {Im}}}
\newcommand{\Ind}{{\mathrm {Ind}}}
\newcommand{\diag}{{\mathrm {diag}}}
\newcommand{\soc}{{\mathrm {soc}}}
\newcommand{\End}{{\mathrm {End}}}
\newcommand{\sol}{{\mathrm {sol}}}
\newcommand{\Hom}{{\mathrm {Hom}}}
\newcommand{\Mor}{{\mathrm {Mor}}}
\newcommand{\Mat}{{\mathrm {Mat}}}
\def\rank{\mathop{\mathrm{ rank}}\nolimits}
\newcommand{\Tr}{{\mathrm {Tr}}}
\newcommand{\tr}{{\mathrm {tr}}}
\newcommand{\Gal}{{\it Gal}}
\newcommand{\Spec}{{\mathrm {Spec}}}
\newcommand{\ad}{{\mathrm {ad}}}
\newcommand{\Sym}{{\mathrm {Sym}}}
\newcommand{\Char}{{\mathrm {char}}}
\newcommand{\pr}{{\mathrm {pr}}}
\newcommand{\rad}{{\mathrm {rad}}}
\newcommand{\abel}{{\mathrm {abel}}}
\newcommand{\codim}{{\mathrm {codim}}}
\newcommand{\ind}{{\mathrm {ind}}}
\newcommand{\Res}{{\mathrm {Res}}}
\newcommand{\Ann}{{\mathrm {Ann}}}
\newcommand{\Ext}{{\mathrm {Ext}}}
\newcommand{\Alt}{{\mathrm {Alt}}}
\newcommand{\AAA}{{\sf A}}
\newcommand{\SSS}{{\sf S}}
\newcommand{\CC}{{\mathbb C}}
\newcommand{\CB}{{\mathbf C}}
\newcommand{\RR}{{\mathbb R}}
\newcommand{\QQ}{{\mathbb Q}}
\newcommand{\ZZ}{{\mathbb Z}}
\newcommand{\NB}{{\mathbf N}}
\newcommand{\ZB}{{\mathbf Z}}
\newcommand{\EE}{{\mathbb E}}
\newcommand{\PP}{{\mathbb P}}
\newcommand{\GC}{{\mathcal G}}
\newcommand{\HC}{{\mathcal H}}
\newcommand{\GA}{{\mathfrak G}}
\newcommand{\TC}{{\mathcal T}}
\newcommand{\SC}{{\mathcal S}}
\newcommand{\RC}{{\mathcal R}}
\newcommand{\bG}{{ \bf G}}
\newcommand\bH{{\bf H}}
\newcommand{\bL} {{\bf L}}
\newcommand{\bM}{{\bf M}}
\newcommand{\bT}{{\bf T}}
\newcommand{\GCD}{\GC^{*}}
\newcommand{\TCD}{\TC^{*}}
\newcommand{\FD}{F^{*}}
\newcommand{\GD}{G^{*}}
\newcommand{\HD}{H^{*}}
\newcommand{\GCF}{\GC^{F}}
\newcommand{\TCF}{\TC^{F}}
\newcommand{\PCF}{\PC^{F}}
\newcommand{\GCDF}{(\GC^{*})^{F^{*}}}
\newcommand{\RGTT}{R^{\GC}_{\TC}(\theta)}
\newcommand{\RGTA}{R^{\GC}_{\TC}(1)}
\newcommand{\Om}{\Omega}
\newcommand{\eps}{\epsilon}
\newcommand{\al}{\alpha}
\newcommand{\chis}{\chi_{s}}
\newcommand{\sigmad}{\sigma^{*}}
\newcommand{\PA}{\boldsymbol{\alpha}}
\newcommand{\gam}{\gamma}
\newcommand{\lam}{\lambda}
\newcommand{\la}{\langle}
\newcommand{\ra}{\rangle}
\newcommand{\hs}{\hat{s}}
\newcommand{\htt}{\hat{t}}
\newcommand{\tn}{\hspace{0.5mm}^{t}\hspace*{-0.2mm}}
\newcommand{\ta}{\hspace{0.5mm}^{2}\hspace*{-0.2mm}}
\newcommand{\tb}{\hspace{0.5mm}^{3}\hspace*{-0.2mm}}
\def\skipa{\vspace{-1.5mm} & \vspace{-1.5mm} & \vspace{-1.5mm}\\}
\newcommand{\tw}[1]{{}^#1\!}
\renewcommand{\mod}{\bmod \,}

\marginparsep-0.5cm

\renewcommand{\thefootnote}{\fnsymbol{footnote}}
\footnotesep6.5pt

\title[Characterisation of nilpotent blocks]
{A characterisation of nilpotent blocks}

\author[Kessar]{Radha Kessar}
\address{School of Engineering and Mathematical Sciences, 
City University London EC1V 0HB, Great Britain}
\email{radha.kessar.1@city.ac.uk}

\author[Linckelmann]{Markus Linckelmann}
\address{School of Engineering and Mathematical Sciences, 
City University  London EC1V 0HB, Great Britain}
\email{markus.linckelmann.1@city.ac.uk}

\author[Navarro]{Gabriel Navarro}
\address{Departament d'\`Algebra, Universitat de Val\`encia,
 Dr. Moliner 50, 46100 Burjassot, Spain.}
\email{gabriel.navarro@uv.es}


\keywords{nilpotent block, height zero, hyperfocal subalgebra}

\subjclass[2010]{Primary 20C20; Secondary}

\begin{abstract} 
Let $B$ be a $p$-block of a finite group, and set $m=$ $\sum \chi(1)^2$, 
the sum taken over all height zero characters of $B$. Motivated by a 
result of M. Isaacs characterising $p$-nilpotent finite groups in terms 
of character degrees, we show that $B$ is nilpotent if and only if the 
exact power of $p$ dividing $m$ is equal to the $p$-part of 
$|G:P|^2|P:R|$, where $P$ is a defect group of $B$ and where $R$ is the 
focal subgroup of $P$ with respect to a fusion system $\CF$ of $B$ on 
$P$. The proof involves the hyperfocal subalgebra $D$ of a source 
algebra of $B$. We conjecture that all ordinary irreducible characters 
of $D$ have degree prime to $p$ if and only if the $\CF$-hyperfocal 
subgroup of $P$ is abelian.
\end{abstract}

\maketitle

\section{Introduction}

Let $p$ be a prime number and let $(K, \CO, k)$ be a $p$-modular system. 
We assume that $k$ is algebraically  closed and that $K$ is a splitting 
field for all finite groups considered in the paper. Let  $G$ be a finite 
group and $B$ a block algebra of $\CO G$.  Let $\CF=\CF_{(P,e_P)}(G, B)$ 
be the fusion system of $ B$ with  respect to a maximal $B$-Brauer pair  
$(P, e_P)$. The focal subgroup of $\CF$ is the subgroup $\foc(\CF)$ 
of $P$ generated by the commutators $[\Aut_\CF(Q),Q]$, where $Q$ runs over
the subgroups of $P$ (see  \cite[I, Def.~ 7.1]{AKO}). Let $\Irr(B)$ denote 
the set of $K$-valued irreducible characters of $B$ and  let $\Irr_{0}(B)$ 
denote the subset of height zero characters. Let $ i \in B^P$ be a source 
idempotent of $B$, and let $S$ be a Sylow $p$-subgroup of $G$ containing 
$P$. 

\begin{thm} \label{degnilp}   
With the notation above, the following are equivalent.
\begin{enumerate}   
\item  $|\sum_{\chi  \in \Irr_{0} (B) } \chi(1) ^2|_p   =  
|S:P|^2 | P :  \foc(\CF)    | $.

\item   $  |\sum_{\chi  \in \Irr_{0} (B) } \chi(i) ^2|_p   =  
| P :    \foc(\CF)  |$.

\item $|\Irr_0(B)|=|P:\foc(\CF)|$.

\item $B$ is nilpotent.
\end{enumerate}
\end{thm} 

Theorem  \ref{degnilp} may be seen as  a block-theoretic analogue of   
Isaacs' result  \cite [Lemma~ 4]{Isa86}  characterizing $p$-nilpotent 
groups via character degrees. The fact that statement (iv) in this 
theorem implies any of (i), (ii), (iii) is an immediate consequence of 
Puig's structure theorem \cite[Theorem 1.6]{Punil} for source algebras 
of nilpotent blocks. It is known that in general the quantity on the 
right hand side of (i) or (ii) always divides the quantity on the left 
hand side  of (i) or (ii), respectively. This is a consequence of the 
free action of $P/ \foc(\CF)$ on the set of height zero characters via 
the $*$-construction (see \cite[Introduction]{Pui00}, \cite{Rob08}).  
The proof of Theorem  \ref{degnilp} relies on the Clifford theoretic 
relationship between the representation theory of the source algebra 
$iBi$ and its hyperfocal subalgebra $D$ (a concept due to Puig; see 
\cite[Theorem 1.8] {Pui00}).  This relationship suggests the following 
`hyperfocal height zero' version of Brauer's height zero conjecture.
Denote by $\Irr (D)$ a set of representatives of isomorphism classes 
of simple $ K \otimes_{\mathcal O}  D$-modules and by $\Irr_0 (D)$ the 
subset corresponding to simple modules of $p'$-degree.  Let $\hyp(\CF)$ 
denote the hyperfocal subgroup of $\CF$. The focal subgroup $\foc(\CF)$ 
is the product  $\hyp(\CF) P'$ of the  hyperfocal subgroup of $\CF$ and 
the derived subgroup $P'$ of $P$; see \cite[Lemma 7.2]{AKO}.  

\begin{con} \label{hyperheightzero}    
With the notation above, assume that $K$ is a splitting field of $D$. 
Then, $\Irr(D) =\Irr_0(D)$ if and only if $ \hyp(\CF)$ is abelian.
\end{con}

Section 2 contains Clifford theoretic considerations regarding
hyperfocal and source algebras. Theorem \ref{degnilp}  is proved in 
Section 3. The last section contains some remarks around 
Conjecture~\ref{hyperheightzero}.

\begin{rem}
For $\chi\in$ $\Irr(B)$ denote by $e_\chi$ the corresponding central 
primitive idempotent in $Z(K\otimes_{\CO} B)$. If $e$ is an idempotent in
$Z(K\otimes_{\CO} B)$, then multiplication by $e$ induces a surjective
$\CO$-algebra homomorphism $B\to Be$; note that $Be$ is an $\CO$-free
quotient algebra of $B$ and an $\CO$-subalgebra of $K\otimes_{\CO} B$ 
but not necessarily a subalgebra of $B$.
The integer $\sum_{\chi\in\Irr_{0} (B)}\ \chi(1) ^2$ is equal to
the $\CO$-rank of the quotient algebra $Be$ of $B$, where
$e=$ $\sum_{\chi\in\Irr_{0} (B)}\ e_\chi$. If $B$ is nilpotent,
then this algebra is Morita equivalent to the commutative
symmetric algebra $\CO P/P'$, where $P'$ is the derived subgroup
of $P$. Indeed, through a Morita equivalence between $B$ and $\CO P$,
the height zero characters in $B$ correspond to the characters
with degree one of $P$, and the intersection of the kernels of these
is $P'$. Okuyama and Tsushima proved in \cite{OkTs} that $B$ is
Morita equivalent to a commutative (and necessarily symmetric)
algebra if and only if $B$ is nilpotent with abelian defect groups.
In a similar spirit one may ask whether there is a characterisation of 
blocks whose quotient $Be$ as defined above is Morita equivalent to a 
commutative symmetric $\CO$-algebra. 
\end{rem}

\begin{rem}
If $|P:\foc(\CF)|\cdot |S:P|^2$ is the exact power of $p$ dividing
$\sum_{\chi\in\Irr_{0} (B)}\ \chi(1) ^2$, then the above theorem
implies that $B$ is nilpotent, hence $\foc(\CF)$ is the
derived subgroup $P'$ of $P$. If $B$ is not nilpotent, then the highest 
power of $p$ dividing  $\sum_{\chi\in\Irr_{0} (B)}\ \chi(1) ^2$ is 
strictly bigger than $|P:\foc(P)|\cdot |S:P|^2$, and one might wonder 
whether it is always at least $|P:P'|\cdot |S:P|^2$. This is, 
however, not the case. Let $p=3$. The group $H=A_4 \times C_4$  has a 
faithful irreducible $\FF_3 H$-module $V$ of dimension $6$.
Let $G=VH$ be the corresponding semidirect product. The group algebra
$\CO G$ has a unique block. If $P$ is a Sylow-$3$-subgroup of $G$, then 
$|P:P'|=27$. This does not divide the sum of the squares of the 
$3'$-degree characters, which is equal to $1548=2^2\cdot 3^2\cdot 43$.
\end{rem}

\section{Extending characters of the hyperfocal subalgebra}

Background material on focal and hyperfocal subgroups of
fusion systems as well as the fusion subsystem $O^p(\CF)$
of the fusion system $\CF$ on $P$, can be found in 
\cite[\S 7.5]{CravenBook}.
We refer to \cite{Pui00} for the notion and basic properties of
hyperfocal subalgebras of source algebras of blocks.
We show in this section that ordinary irreducible characters
of the hyperfocal subalgebra of degree prime to $p$ extend
to the source algebra in precisely $|P:\foc(\CF)|$ ways.
The key ingredient is a special case of a result of Diaz, Glesser,
Park and Stancu \cite{DGPS11}  which extends to fusion  systems work 
of Gagola and Isaacs \cite{GaIs08} on the transfer homomorphism.
For a saturated fusion system ${\mathcal G} $ on a finite $p$-group  $S$,
we denote by $  \tau_{S, \Omega} ^{\mathcal G}   :  S \to S/[S,S] $  the   
transfer  map with respect to some characteristic element $\Omega $ for 
${\mathcal F} $ (see \cite[Definition~2.5]{DGPS11}), and by 
$ T_{\mathcal G} $  the subgroup of $S$ containing $[S,S]$ such that  
$T_{\mathcal G} /[S,S] =$ $\Im(\tau_{S, \Omega}^{\mathcal G}) $. By   
\cite[Lemma~2.6]{DGPS11}, $T_{\mathcal G}$ is independent of the choice 
of $\Omega $.  Further, if $ U$ is a subgroup of $S$ containing 
$\hyp(\mathcal G)$, we denote by ${\mathcal G}_U $, the  unique 
saturated subsytem of $ {\mathcal G}$  on $U$ of $p$-power index 
(see \cite[Theorem~I.7.4]{AKO}).
 
\begin{pro} \label{hypersplitpre}  \cite[Proposition~5.3]{DGPS11}
 Let ${\mathcal G} $ be a saturated fusion system on  a finite 
$p$-group $S$ and let $U$ be a normal subgroup of $S$ containing 
$\hyp(\mathcal G)  $. If  $V$ is a subgroup of $U$  containing 
$T_{{\mathcal G_U}} [U,S]$, then  
$ S/V \cong U/V \times T_{\mathcal G}V/V $. 
\end{pro}

\begin{cor}\label{hypersplit}   
Let ${\mathcal G}$ be a saturated fusion system on a finite $p$-group 
$S$. Then  $\hyp(\mathcal G)/[\hyp(\mathcal G), S]$ is a direct 
factor of $S/[\hyp(\mathcal G),S ] $.  
\end{cor}

\begin{proof}  
Let   $U= \hyp(\mathcal G)$.  By the minimality  of $ {\mathcal G_U }$    
as  normal saturated subsytem of ${\mathcal G } $  of $p$-power index, 
we have that  $ \foc({\mathcal G_U })= \hyp({\mathcal G}_U)  = U$ 
(see \cite [Theorem~7.53]{CravenBook}  and \cite[Lemma~ I.7.2]{AKO}).       
By  Proposition 2.7  of \cite{DGPS11} and the first line of its proof   
we have that 
$$ U/[U,U] \cong    \foc({\mathcal G_U })/[U,U] \times  
T_{\mathcal G_U} /[U,U] . $$  
Since $\foc({\mathcal G_U }) = U $, it follows that 
$T_{\mathcal G_U}=[U,U] $. The corollary follows from 
Proposition~\ref{hypersplitpre}  applied with $V  = [U,S] $. 
\end{proof}

Keep the previous notation. In addition, set $Q=$ $\hyp(\CF)$ and 
$R=\foc(\CF)$, and set $A=iBi$, the source algebra of $B$ corresponding 
to the defect group $P$ of $B$ and the source idempotent $i$. By 
\cite[Theorem 1.8] {Pui00} the hyperfocal subalgebra $D$ is the 
unique (up to $(A^P)^{\times}$-conjugation)  $P$-stable unitary 
subalgebra of $A$ such that 
$$D \cap Pi =  Qi   \  \   \  {\rm and } \  \  \    A =  
D \otimes_{\CO Q} \CO P =  \oplus_{u \in P/Q}  Du. $$
In particular,  $D$ is a $Q$-interior $P$-algebra. We denote by 
$\hat A =K\otimes _{\CO}  A$ and  $\hat D =K \otimes _{\CO}  A$  the  
$K$-algebras obtained by extensions of scalars to $K$. The $K$-algebras 
$\hat A$ and $\hat D$ are semi-simple; indeed, for $\hat A$ this follows 
from the Morita equivalence between $A$ and $B$, and for $\hat D$ this 
is a consequence of the fact that $J(\hat D)\hat A$ is an ideal in 
$\hat A$, hence nilpotent, hence zero. By replacing $K$ by a suitable  
finite extension, we  may and will assume that $\hat A$ and $\hat D$ are     
also split. We denote by $\Irr(A)$ (respectively ($\Irr (D)$) a set 
of representatives of isomorphism classes of simple 
$\hat A$-modules (respectively $\hat D$-modules).   

\medskip
The   $Q$-interior $P$-algebra  structure of $D$ allows the usual 
notions of Clifford theory  to be carried over to the inclusion of   
$\hat D$ in $\hat A$ (see  \cite{Fan09}). In particular, $P$ acts 
on $\Irr (D)$. For  $W\in\Irr(D) $ and  $ u\in P $  we denote by  
$ \, ^u W\in\Irr(D) $ the  image of $W$ under the action of $u$ and  
by  $I_P(W)$   the  stabiliser of $W$ in $P$. Note that 
$Q \leq I_P(W)$.  We say that  a simple $\hat A$-module  {\it covers} $W$   
if $W$  appears as a composition factor  of  $\Res_{\hat D} (V)$ and 
denote by  $\Irr(A| W )$ the subset of $\Irr (A)$  which cover $W$.
For $W\in \Irr (D)$ denote by $e_V$ the primitive central idempotent 
of $\hat D $ corresponding to $W$ and similarly for  $\hat A$. Note that 
$ \hat D e_{W} \cong \Mat_{\dim_K (W)} (K)$.  

Suppose that   $W \in \Irr (D)$  is $P$-stable. Then  
$\hat D e_{W} \cong \Mat_{\dim_K (W)} (K)$  is a $P$-algebra and  by 
the Noether-Skolem theorem, every automorphism of $\hat D e_{W}$ is  
inner. Let $I$  be a set of coset representatives  of $Q$ in $P$  
containing $1$. For each  $y \in  I$, choose  an  element $ s_y \in$
$(\hat D e_W){\times}$ such that $s_y d s_{y^{-1}}   = \,^ yd $ for all 
$d \in \hat D$. For each $x \in Q $ and $  y \in I $, set 
$s_{xy}=xs_y $.  Then, for all $x, y \in P$,  we have 
$s(x)s(y) s(xy)^{-1} \in  Z ( \hat D e_{W} )^{\times} $.  Identifying  
$Z ( \hat D e_W ) $  with $K$, the map 
$ \beta :   P \times P   \to  K^{\times} $ given by 
$ \beta (y, y') =  s_{yy'}^ {-1}  s_y s_{y'}  $,  where $y, y' \in P$,  
is a $2$-cocycle such that $\beta (xy, x'y') =\beta(y, y') $ for  all  
$x, x' \in Q $.   So, $\beta$ is the restriction to $P$ along the 
canonical map $P\to$ $P/Q$ of a $2$ -cocycle 
$\bar \beta : P/Q \times P/Q \to K ^{\times}$. We denote by $\alpha_W$ 
the image of $ \bar \beta   $ in $H^2(P/Q, K^{\times}) $.  Then   
$\alpha_W $ is independent of the choice of $I$ and the choice  of 
elements $s_y $, $y\in  I$.  

For $L$  a finite group  and $\alpha \in H^2(L, K^{\times}) $, we 
denote by $K_{\alpha} L$  a twisted  group  algebra of $L$   over 
$K$ corresponding to $\alpha$. The following proposition is a 
rewording of  some of the results in  \cite{Fan09}.

\begin{pro} \label{clifford} 
Let $ V \in \Irr (A) $ and  $ W  \in \Irr (D) $.
\begin{enumerate}
\item  $\Res_{\hat D} (V)$ is semi-simple,  and if $V$ covers $W$, then 
$$ \Res_{\hat D} (V) \cong   m_{V}  \sum_{u \in P/I_P(W)} \,^uW . $$

\item $ V$ covers $W$ if and only if $e_We_V \ne 0 $.

\item  Suppose that $W\in\Irr(A)$ is $P$-stable. Then   
$$\hat A  e_W \cong \hat D e_W \otimes _K K_{\alpha_W {-1}}(P/Q)$$  
as $K$-algebras.  In particular, there is a bijection $V \to V_0$   
between $\Irr( A| W ) $ and $ \Irr(K_{\alpha_W} ( P/Q) )$   such 
that $\dim_K(V) =$ $\dim_K(W) \dim_K(V_0)$.

\item  Suppose that $ V$ covers $W$. If $\dim_K(V)$ is 
prime to $p$ then $W$ is $P$-stable, $\dim_K(W)=$ $\dim_K(V)$
is prime to $p$ and $\alpha_W$ is trivial.

\item   Suppose that $W$ is $P$-stable, $\dim_K(W) $ is prime to $p$,    
and  $\alpha_W $ is trivial.  Then there are $[P:R]$ elements in  
$\Irr(A|W)$ whose dimension is prime to $p$ and they all have the 
same degree $\dim_K(W)$.
\end{enumerate}
\end{pro} 
 
\begin{proof}   See   the proof of  \cite[Theorem~1.6] {Fan09}. 
\end{proof}

Corollary \ref{hypersplit} provides the following key extendibility result.

\begin{pro} \label{extend} 
Suppose that $ W \in \Irr (D)$ is  $P$-stable and that  $\dim_K(W)$ is 
prime to $p$. Then $\alpha_W $ is trivial. Equivalently, $W$ extends
to $|P:R|$ pairwise nonisomorphic simple $\hat A$-modules.
\end{pro}

\begin{proof}   
Let $\delta : Q \to K^{\times}$ be defined by $\delta (x)= $
$\det(x e_W )$, where $x\in Q$. Then $\delta $ is a $P$-stable
group homomorphism. Since every element of $P$ acts as an inner 
automorphism  of $\hat D e_W $,  $[Q,P] \leq \Ker (\delta) $.    
Thus, by Corollary \ref{hypersplit},  $\delta $ extends to a group 
homomorphism   from $P$ to $K^{\times}$.  Choose such an extension, 
and  by abuse of notation, denote it again by $\delta$. Let $I$ be a 
set of coset representatives   of $Q$ in $P$.   For each $y \in I$, 
choose the element $s_y$ in $(\hat De_W)^{\times}$  as in the 
definition of $\alpha_W $  such that  $\det(s_y ) =$ $\delta (y) $.  
This is possible  as the dimension of $W$ is prime to $p$  and  
$\delta (y) $ is a $|P|$-th root of unity (and since  we may assume 
that $K$ contains a $|P|$-th root of unity). Then for all $ x\in P$, 
$\det(s_x) =$ $\delta (x) $ and consequently, $\det(s_{xy}) =$
$\det(s_x) \det(s_y) $  for all $x, y \in P$.  By 
the definition of $\beta $, we have $s_xs_y =\beta (x, y) s_{xy} $ for 
all $x, y \in P$.  Hence , for all $x, y\in P $, we have 
$\beta(x,y)^n =1 $,  where $n= \dim_K(W) $. Consequently,  
$\alpha_W  \in  H^2(P/Q, K^{\times} )$  has order dividing $n$. Since 
$n$ is prime to $p$  and $H^2(P/Q, K^{\times} )$ is a $p$-group,  it 
follows that $\alpha_W =1$, as  claimed.   
\end{proof}

\section{Proof  of Theorem \ref{degnilp}}
 
\begin{proof}[{Proof of Theorem \ref{degnilp}}]     
As before, we denote by $A=$ $iBi$ the source algebra of
$B$ with respect to the defect group $P$ and the source
idempotent $i\in$ $B^P$. We denote by $\CF$ the fusion system
of $B$ on $P$ determined by $i$. We set $R=$  $\foc(\CF)$ and
$Q=$ $\hyp(\CF)$.
We first prove the equivalence of (i) and (ii). Recall from
\cite[Introduction]{Pui00} or \cite{Rob08} that the  Brou\'e-Puig  
$*$-construction from \cite{BrPuloc} on characters yields a free action 
of $P/R $ on  $\Irr_0(B) $. Let $\chi, \chi' \in \Irr_0(B)$.   
If $\chi, \chi' $ lie in the same  orbit of  $P/R$, then  from the 
definition of the $*$-construction, it is  immediate 
that $ \chi(x) =\chi'(x) $  for all $p$-regular elements   of $G$, hence  
$\chi(j)=\chi'(j) $ for any idempotent $j$ of $\CO G$  
(as $\chi$ and $\chi'$  have the same decomposition numbers).
In particular, if $\chi $ and $\chi'$ are in the same 
orbit of $P/R$ on $\Irr_0(B) $, then $\chi(1) =\chi'(1) $ and 
$\chi(i)=\chi'(i)$. Thus, letting ${\mathcal S} $ denote a set of 
representatives of the $P/R$-orbits on $\Irr_0(B)$, we have that 
$$  \sum_{ \chi \in \Irr_0 (B) }  \chi(1) ^2      =    
|P:R|\sum_{ \chi \in {\mathcal S} }  \chi(1)^2; $$
$$  \sum_{ \chi \in \Irr_0 (B) }  \chi(i) ^2      =    
|P:R|\sum_{ \chi \in {\mathcal S} }  \chi(i)^2.$$
Since $|S: P | $ divides $\chi (1)  $  for all $\chi \in \Irr (B) $,    
it follows from the above that in order to prove the equivalence of
(i) and (ii), it suffices to prove that  the integer  
$\sum_{ \chi \in {\mathcal S}} \chi(i)^2$ belongs to $J(\CO )$  
if and only if the integer 
$\sum_{ \chi \in {\mathcal S}}(\frac{\chi(1)}{| S:P|})^2$   
belongs to $J(\CO )$. 

Let $\omega_B: Z(B) \to Z(B)/J(Z(B))=k$ be the canonical 
surjection. By \cite{PiPu} or \cite[9.3]{ThevdualityG},  
$\Tr_{P}^{G} (i) $ is an invertible element of $Z(B)$, hence  
$\omega_B (\Tr_{P}^{G}(i)) \ne 0$.
Recall that  for any $ z\in Z(B)$, and any $\chi\in\Irr (B)$, 
we have $\frac{\chi(z)}{\chi(1)} \in {\mathcal O}$ and    
$\omega_B(z) =\overline {\frac{\chi(z)  }{\chi(1)}}$, where for 
$x\in\CO $ we denote by $\bar x$ the image in $k$ of $x$ under the 
canonical surjection $ \CO/J(\CO)$.
Thus, for any $\chi \in \Irr  (B) $, we have 
$$ 0\ne   \omega_B (\Tr_{P}^{G}(i)) = 
\overline {\frac{\chi(\Tr_P^G(i))}{\chi(1)}}     =    
\overline {|G: P|\frac{ \chi (i)}{ \chi(1)}}. $$
Choose  $\lambda \in \CO^{\times} $ with $\bar\lambda =  
\omega_B (\Tr_{P}^{G}(i))$.
The above shows that for any $\chi\in\Irr(B)$, there exists    
an element $t_{\chi}\in J(O)$ such that 
$$|G: P| \frac{ \chi (i)}{ \chi (1) } = \lambda   +  t_{\chi} . $$
Since, $|G:S| $ is  invertible in $\CO $,  by suitably replacing 
$\lambda $ and $t_{\chi} $,  we  obtain that   there exists 
$\lambda \in \CO^{\times}$ such that for all $\chi \in \Irr (B) $, there 
exists $t_{\chi} \in J(O)$ such that 
$$|S: P| \frac{ \chi (i)}{ \chi (1) } = \lambda + t_{\chi}\ . $$
Since $t_{\chi }  \in J(\CO) $, we have that   
$$  \sum_{ \chi \in {\mathcal S} }  \chi(i) ^2  =    
\sum_{ \chi \in {\mathcal S} } (\lambda + t_{\chi})^2  
(\frac{\chi(1)}{ |S:P|})^2   \equiv  
\lambda^2 \sum_{ \chi \in {\mathcal S}}  (\frac{\chi(1)}{|S:P|})^2 
\mod (J(\CO)).  $$ 
The equivalence of (i) and (ii) follows. If (iii) holds, then
by the above, the height zero characters in $B$ form a unique
regular $P/R$-orbit with respect to the $*$-construction;
in particular, all height zero characters of $B$ have the same degree
with $p$-part $|S:P|$, and hence (iii) implies (i). We prove next that 
(ii) implies (iv).   
The  strategy is the same as used for \cite[Lemma 4]{Isa86}.  Let 
${\mathcal V}$ denote the subset of $\Irr_0 (D) $  consisting of 
$P$-stable elements.
Note that $\chi\in\Irr_0(B)$ is of height zero if and only if 
$\chi(i)$ is prime to $p$. Moreover, $\chi(i)$ is the 
dimension of the simple $\hat A$-module corresponding to $\chi$ 
under the canonical  Morita equivalence between $B$ and $A$. By 
Proposition~\ref{clifford} (i)-(iv)  and Proposition~\ref{extend} 
it follows that   
$$  \sum_{\chi  \in \Irr_{0} (B) } \chi(i) ^2  =  
|P:R|   \sum_{ W \in {\mathcal V}}  \dim_K(W)^2.$$
In particular, $|P:R|$ divides  
$\sum_{\chi  \in \Irr_{0} (B) } \chi(i) ^2$ and we have 
$$\frac{1}{|P:R|}  \sum_{\chi  \in \Irr_{0} (B) } \chi(i) ^2  =    
\sum_{ W \in {\mathcal V}}  \dim_K(W)^2.$$
If $ W' \in \Irr(D) \setminus {\mathcal V} $, then either $p$ 
divides the dimension of $W'$ or the size of the  
$P$-orbit of $W'$ is divisible by $p$. Hence,
$$\sum_{ W \in \mathcal  V} \dim_K(W)^2 \equiv   
\sum_{W\in\Irr (D) } \dim_K(W)^2     \    (\mod     p).  $$  
Assume that $\frac{1}{|P:R|}  
\sum_{\chi\in \Irr_{0}(B)}\ \chi(i) ^2$ is a $p'$-integer.   
Then by the above, $\dim_K(\hat D) =$ 
$\sum_{W\in\Irr(D) } \dim_K(W)^2$ is a $p'$-integer. The 
hyperfocal subalgebra $D$ is a direct summand of $A$ as a left 
$D$-module, hence projective as a left $\CO Q$-module. In particular, 
$|Q|$ divides the $\CO$-rank of $D$. Thus $Q=1$, which means that   
$\CF$ and hence $B$ are nilpotent. This proves that (ii) implies (iv). 
As mentioned earlier, the fact that (iv) implies any of (i), (ii), 
(iii)  is immediate from the structure theorem 
\cite[Theorem 1.6]{Punil} for source algebras of nilpotent blocks.
\end{proof} 

\section{Around Conjecture~\ref{hyperheightzero}.}

As before, $G$ is a finite group, $B$ a block algebra of $\OG$
with corresponding block idempotent $b=$ $1_B$, and $P$ a defect group 
of $B$. Let $i\in$ $B^P$ be a source idempotent and $A=$ $iBi$ the
corresponding source algebra. 

\begin{pro}
Suppose $P$ is abelian. Then $\Irr_0(D)= \Irr (D)$.
\end{pro}

\begin{proof}     
This follows from the forward direction of the height zero conjecture
(proved by Kessar and Malle in \cite{KeMa}),  and Proposition 
\ref{clifford}.     
\end{proof}

\begin{pro}
Conjecture~\ref{hyperheightzero} holds for blocks with a normal 
defect group.
\end{pro}

\begin{proof}  
Suppose that the defect group $P$ of $B$ is normal in $G$.
By K\"ulshammer's structure theorem in \cite{Kuenormal}, the source 
algebra $A$ is isomorphic to a twisted group algebra of $P\rtimes E$, 
where $E$ is the inertial quotient of $B$.  From this it follows that 
a hyperfocal source algebra is isomorphic to a twisted group algebra 
of $Q \rtimes E$. The result follows.
\end{proof}

The block $B$ is said to be {\it of principal type} if for any
subgroup $Q$ of $P$, the idempotent $\Br_Q(b)$ is a block of
$kC_G(Q)$, or equivalently, if the algebra $B(Q)=$ 
$kC_G(Q)\Br_Q(b)$ is a block algebra of $kC_G(Q)$. In that case,
we have $e_P=$ $\Br_P(b)$, there is a unique local point $\gamma$
of $P$ on $B$, and the fusion system $\CF$ of $B$ on $P$ is equal to 
the fusion system $\CF_P(G)$ induced by conjugation of elements in $G$ 
on subgroups of $P$. The principal block of
$\OG$ is of principal type by Brauer's third main theorem. 
The following two observations are well-known and included for 
convenience; the proofs are routine block theory. Recall that the 
multiplicity $m_\gamma$ of a local point of $P$ on $B$ is the number 
$|\gamma\cap I|$ of elements of $\gamma$ which appear in a primitive 
orthogonal idempotent decomposition $I$ of $b$ in $B^P$.

\begin{lem}\label{multprinctype} 
Suppose that the defect group $P$ of $B$ is a 
Sylow-$p$-subgroup of $G$. Let $\gamma$ be a local point of 
$P$ on ${\CO G}^P $.  Then the multiplicity  $m_{\gamma}$  of 
$\gamma$ on $B$ is prime to $p$ and divides the order of $C_G(P)$.
\end{lem}

\begin{proof}   
Let $e$ be the unique block of $kC_G(P)$ such that
$\Br_P(\gamma)\subseteq$ $kC_G(P)e$. Then $\Br_P(\gamma)$ is the 
unique point of $kC_G(P)e$. Thus $m_\gamma$ is the dimension
of the unique (up to isomorphism) simple $kC_G(P)e$-module
$V$. Since $P$ is a Sylow-$p$-subgroup of $G$, it follows that
$C_G(P)=$ $Z(P)\times C'$ for some $p'$-subgroup $C'$ of $C_G(P)$,
and hence $e=$ $e_\eta$, where  $\eta$ is an irreducible character
of $C'$ with coefficients in $k$  and $e_\eta$ is the corresponding
central primitive idempotent in $kC'$ (which makes sense as $|C'|$ is
prime to $p$). Thus $\dim_k(V)=$ $\eta(1)$ divides $|C'|$, whence
the result.
\end{proof}

\begin{pro} \label{princtypehyper}  
Suppose that $B$ is of principal type and that $P$ is a 
Sylow-$p$-subgroup of $G$. Set $N=$ $O^p(G)$. Then $b$ is a block 
idempotent of $\CO N$. Setting $C=$ $\CO Nb$, the following hold.

\smallskip\noindent (i)
The block $C$ of $\CO N$ is of principal type, with a defect group 
$Q=$ $P\cap N=$ $\hyp(\CF)$ and fusion system $O^p(\CF)$.

\smallskip\noindent (ii) 
The algebra $C^P$ contains a $(B^P)^\times$-conjugate of $i$.

\smallskip\noindent (iii)
If $i\in$ $C^P$, then $D=$ $iCi$ is a hyperfocal subalgebra
of the source algebra $A=$ $iBi$ of $B$.

\smallskip\noindent (iv)
If $i\in$ $C^P$, then there is a source idempotent $j\in$ $(iCi)^Q$
of the block $C$ belonging to the unique local point $\delta$ of
$Q$ on $C$. The multiplicity $m_\delta^\gamma$ of $\delta$ on
$iCi$ is prime to $p$, and for any $\chi\in$ $\Irr(C)$ we have
$\chi(i)\equiv m^\gamma_\delta \chi(j)\mod\ p$. 

\smallskip\noindent (v)
All irreducible characters of $iCi$ have degree prime to $p$ if
and only of all irreducible characters of $jCj$ have degree
prime to $p$.
\end{pro}

\begin{proof} 
Since a block idempotent is supported on $p'$-elements, we have
$b\in$ $(\CO N)^G$. Thus $b$ is a $G$-conjugacy class sum of
blocks of $\CO N$. As $G=$ $NP$, the idempotent $b$ is in fact
a $P$-conjugacy class sum of blocks of $\CO N$. Since $\Br_P(b)\neq$
$0$, this implies that $b$ is a block of $\CO N$. Then
$Q=$ $P\cap N$ is a Sylow-$p$-subgroup of $N$, hence a defect
group of the block algebra $C=$ $\CO Nb$ of $\CO N$.
By Puig's hyperfocal subgroup theorem, we have $Q=$ $\hyp(\CF)$. 
Let $R$ be a subgroup of $Q$. We need to show that $f=$ $\Br_R(b)$
is a block of $kC_N(R)$. After replacing $R$ by a conjugate,
if necessary, we may assume that $C_P(R)$ is a Sylow-$p$-subgroup
of $C_G(R)$. Since $C_N(R)$ is normal of $p$-power index in
$C_G(R)$, it follows that $C_G(R)=$ $C_N(R)C_P(R)$. We have
$\Br_{C_P(R)}(f)\neq$ $0$ because $f$ is a block of $kC_G(R)$
with defect group $C_P(R)$. Arguing as above for $b$ it follows
that $f$ is a $C_P(R)$-conjugacy class of block idempotents
of $kC_N(R)$, and hence the condition $\Br_{C_P(R)}(f)\neq$ $0$ 
implies that $f$ remains a block idempotent of $kC_N(R)$. Thus the 
block $C$ is of principal type, and hence $C$ has the fusion system
$\CF_{Q}(N)$, which is equal to $O^p(\CF)$. This proves (i).
Statement (ii) follows as the inclusion $C\to$ $B$ 
is a semi-covering (cf. \cite[3.9, 3.16]{KP}) and therefore any 
primitive (local) idempotent in $C^P$ remains primitive (local) in 
$B^P$. Thus the unique local point $\gamma$ of $P$ on $B$ contains 
an element in $C$. Choosing such an $i\in$ $\gamma\cap C$, it
follows that $iCi$ is a unitary $P$-stable subalgebra of $iBi$ such 
that $iCi \cap P =$ $i(P\cap N)i  = Qi$. Statement (iii) follows by 
the uniqueness of hyperfocal subalgebras. By \cite[4.2]{Pui00},
the group $Q$ has a unique local point $\delta$ on $D=$ $iCi$.
This is then necessarily a defect pointed group of $C$, and hence
any $j\in$ $\delta$ is a source idempotent of $C$. 
If $j'$ belongs to a nonlocal point of a subgroup $R$ of
$P$ on $\OG$, then $j'$ can be written as a trace from a proper
subgroup of $R$, and hence $\chi(j')\in$ $J(\CO)$ for any $\chi\in$
$\Irr(G)$. Thus, if $\chi\in$ $\Irr(B)$, then $\chi(1)\equiv$
$m_\gamma\chi(i)\mod\ J(\CO)$, since $\gamma$ is the unique local
point of $P$ on $B$. Similarly, we have
$\chi(i)\equiv m^\gamma_\delta \chi(j) \mod\ J(\CO)$. Since $P$ is
a Sylow-$p$-subgroup of $G$, it follows that the height zero characters
in $B$ are the characters of degree prime to $p$ in $B$. Thus
the previous lemma and the above congruence applied to a height zero 
character yield together that  $m^\gamma_\delta$ is prime to $p$.
This proves (iv).  Multiplication
by $j$ induces a Morita equivalence between the hyperfocal subalgebra
$iCi$ of $B$ and the source algebra $jCj$ of $C$. In particular,
multiplication by $j$ induces a bijection between irreducible
characters of $iCi$ and $jCj$. Thus (v) follows from (iv).
\end{proof} 

\begin{pro} \label{hyperheightsolvable}
Conjecture \ref{hyperheightzero} holds for blocks of finite
$p$-solvable groups.
\end{pro}

\begin{proof}
Suppose that $G$ is $p$-solvable. Then $B$ is source algebra equivalent 
to a block of principal type of a subgroup $H$ of $G$ such that $P$ is a 
Sylow-$p$-subgroup of $H$ and such that the block idempotent belongs to 
$\CO O_{p'}H$; this is a version of Fong-Reynolds reduction, as 
presented, for instance, in \cite[Theorem 5.1]{HaLi}. Thus we may assume 
that $B$ is of principal type and that $P$ is a Sylow-$p$-subgroup of 
$G$. Set $N=$ $O^p(G)$. By \ref{princtypehyper},  the block idempotent 
$b$ of $B$ remains a block idempotent in $\CO N$, and the group $Q=$ 
$N\cap P=$ $\hyp(\CF)$ is a defect group of $C=$ $\CO Nb$. It follows 
from \ref{princtypehyper} (v) that all irreducible characters of a 
hyperfocal subalgebra $D$ of $B$ have degree prime to $p$ if and only if 
all irreducible characters in $C$ have degree prime to $p$. Since by a 
result of Gluck and Wolf \cite{GluckWolf} the height zero conjecture 
holds for blocks of $p$-solvable groups, it follows that this is 
equivalent to $Q$ being abelian.
\end{proof}

Since \cite{KeMa} and \cite{GluckWolf} invoke the classification of 
finite simple groups, so do the above proofs of Conjecture 
\ref{hyperheightzero} for blocks with an abelian defect group or
blocks of $p$-solvable finite groups.

\begin{pro}  \label{princheightzerohyper}   
For  blocks of principal type and maximal defect, Brauer's
height zero conjecture implies Conjecture \ref{hyperheightzero}.
\end{pro}

\begin{proof}     
Suppose that Brauer's height zero conjecture holds for blocks of 
principal type and maximal defect. Suppose that $B$ is of principal 
type and that its defect group is a Sylow-$p$-subgroup of $G$. Let $N$, 
$C$, $i$, $j$ be as in  Proposition~\ref{princtypehyper}. By
Proposition~\ref{princtypehyper} (v), all irreducible characters
of the hyperfocal algebra $D=$ $iCi$ have degree prime to $p$
if and only if all irreducible characters of the source algebra
$jCj$ of $C$ have degree prime to $p$, which by Brauer's height
zero conjecture (assumed to be true) applied to the block $C$
is true if and only if $Q$ is abelian. 
\end{proof}

\begin{rem}
The converse implication in \ref{princheightzerohyper}  holds
for the `forward' direction of the two conjectures. 
More precisely, suppose that Conjecture \ref{hyperheightzero} holds 
for all blocks of principal type and maximal defect. Suppose that 
$B$ is of principal type and that $P$ is a Sylow-$p$-subgroup of $G$.
Let $N$, $C$, $i$, $j$ be as in  Proposition~\ref{princtypehyper}. 
It follows from \ref{clifford} that all irreducible characters of $iBi$ 
have $p'$-degree if an only if all irreducible character of $iCi$ are 
$P$-stable, have $p'$-degree, $P/Q$ is abelian, and all $2$-cocycles 
$\alpha_W$ as in \ref{clifford} are trivial. Since we assume 
\ref{hyperheightzero} to be true, this is equivalent to requiring that 
all irreducible characters of $iCi$ are $P$-stable, the groups $Q$, 
$P/Q$ are both abelian, and all $\alpha_W$ as before are trivial. 
We do not know whether these conditions force $P$ to be abelian. We can 
show that these conditions are satisfied if $P$ is abelian. Indeed, if 
$P$ is abelian, then $Q$ and $P/Q$ are trivially abelian. By a theorem 
of Kn\"orr in \cite{Knorr}, every $\CO$-free $B$-module affording an 
irreducible character $\chi$ has vertex $P$. But then any character of 
$C$ covered by $\chi$ must be $P$-stable (because otherwise one could 
write $\chi$ as being induced from a proper subgroup or $G$ containing 
$N$, hence not containing $P$, yielding a lattice with vertex smaller 
than $P$). All characters of $iCi$ have $p'$-degree, so all $\alpha_W$ 
are trivial by \ref{extend}.
\end{rem}

\end{document}